\documentclass[12pt,oneside]{amsart}

\usepackage[margin=1.1in]{geometry}
\usepackage{amsmath,amsfonts,amssymb,amsthm,mathtools}
\usepackage{microtype}
\usepackage[colorlinks=true,linkcolor=blue,citecolor=blue,urlcolor=blue]{hyperref}
\numberwithin{equation}{section}

\newtheorem{theorem}{Theorem}[section]
\newtheorem{proposition}[theorem]{Proposition}
\newtheorem{lemma}[theorem]{Lemma}

\theoremstyle{definition}
\newtheorem{definition}[theorem]{Definition}
\theoremstyle{remark}
\newtheorem{remark}[theorem]{Remark}

\newcommand{\Z}{\mathbb Z}

\newcommand{\cU}{\mathcal U}

\title[Growth gaps and generating sets]{Growth gaps and generating sets}
\author{Aleksander Skenderi}
\address{Department of Mathematics, Jagiellonian University, Krak\'ow, Poland}
\email{askenderi@wisc.edu}

\author{Gal Yehuda}
\address{Department of Mathematics, Yale University, New Haven, CT, USA}
\email{gal.yehuda@yale.edu}

\subjclass[2020]{20F67, 20E07, 22E40}
\keywords{growth of subgroups, growth gap, hyperbolic groups,
amenable actions, lattices}
\date{}

\begin{document}

\begin{abstract}
We show that the existence of a growth gap for infinite-index subgroups of a given finitely genrated group can
depend on the finite generating set. 
More precisely, for any irreducible lattice $\Lambda$ in a higher rank semisimple Lie group $G$ with Kazhdan's property (T), the group $\Lambda \times \Lambda$ admits one finite symmetric generating set with a growth gap and another without a growth gap. 
We also prove that the growth gap can be made arbitrarily small.
In contrast, for a non-elementary hyperbolic group the existence of a growth gap is independent of the finite generating set.
\end{abstract}

\maketitle

\section{Introduction}
A guiding philosophy in group theory is the following: to understand a group, one should understand how it acts on naturally related sets and spaces. If $(X,d)$ is a metric space and $\Gamma$ is a group acting properly discontinuously by isometries on $X$, then, since any compact subset intersects any of its $\Gamma$-translates only finitely many times, it is natural to wonder how quickly the $\Gamma$-orbits in $X$ ``escape to infinity.'' More precisely, one wishes to understand the \emph{exponential growth rate of $\Gamma$ with respect to $X$}, which is defined as
\begin{align*}
\omega(\Gamma, X) = \limsup_{n \to \infty} \frac{1}{n} \log \big| \{\gamma \in \Gamma : d(o, \gamma o) \leq n \} \big|,
\end{align*} where $o$ is a point of $X$. Notice that, as $\Gamma$ acts on $X$ by isometries, this definition is independent of the choice $x \in X$, hence is well-defined. However, $\omega(\Gamma, X)$ does depend on the choice of space $X$, with $\omega(\Gamma, X)$ potentially revealing much about the structure and properties of the group $\Gamma$ depending on the choice of space $X$. For instance, if $\Gamma$ is a non-elementary discrete subgroup of isometries of the real hyperbolic space $X = \mathbb{H}^{n}$, then by the celebrated theorem of Bishop--Jones \cite{BishopJones1997}, $\omega(\Gamma,X)$ also measures the Hausdorff dimension of the conical limit set $\Lambda_{\mathrm{con}}(\Gamma) \subset \partial \mathbb{H}^{n}$ of $\Gamma$.
\par 
For certain groups $\Gamma$ and well-chosen spaces $X$, the exponential growth rate $\omega(\Delta, X)$ of a subgroup $\Delta < \Gamma$ reveals important group theoretic properties of $\Delta$, such as whether it is of finite or of infinite index in $\Gamma$. For instance, the celebrated gap theorems of Corlette \cite{Corlette1990} and Leuzinger \cite{Leuzinger2003} (see also related work of Quint \cite{Quint2003}) state that if $\Gamma$ is a lattice in a semisimple Lie group $G$ with Kazhdan's property (T), and if $X$ is the Riemannian symmetric space of $G$, then there exists a constant $\epsilon = \epsilon(G) > 0$, depending only on the Lie group $G$, so that a subgroup $\Delta < \Gamma$ satisfies $\omega(\Delta, X) \leq \omega(\Gamma, X) - \epsilon$ if and only if $[\Gamma : \Delta] = \infty$ (and, therefore, $\omega(\Delta,X) = \omega(\Gamma, X)$ if and only if $[\Gamma : \Delta] < \infty$).
\par 
In light of the theorems of Corlette and Leuzinger, it is natural to wonder whether similar results persist for other groups $\Gamma$ or spaces $X$. In this article, we will be primarily interested in finitely generated groups, so let $\Gamma$ be such a group and let $S$ be a finite symmetric generating set.
We write $|g|_S$ for the corresponding word length and
\[
  B_S(n)=\{ \gamma \in \Gamma : |\gamma|_S\le n\}.
\]
For a subset $A\subseteq \Gamma$, its exponential growth exponent with respect to
$S$ is
\[
  \omega_S(A)
  =
  \limsup_{n\to\infty}\frac1n\log |A\cap B_S(n)|.
\]
In particular, if $\Delta<\Gamma$ is a subgroup, then $\omega_S(\Gamma)$ is computed using
the restriction to $\Delta$ of the ambient word metric on $\Gamma$, rather than a word
metric coming from a generating set of $\Delta$. If $X$ denotes the Cayley graph of $\Gamma$ with respect to $S$, then in terms of the previous notation, we have $\omega(\Delta,X) = \omega_{S}(\Delta)$. The following is the key definition of this paper.
\begin{definition}
The \emph{top growth gap} of $(G,S)$ is
\[
  \operatorname{Gap}_S(G)
  =
  \omega_S(G)
  -
  \sup_{\substack{H<G\\ {}[G:H]=\infty}}\omega_S(H).
\]
We say that $(G,S)$ \emph{has a growth gap} if
$\operatorname{Gap}_S(G)>0$.
\end{definition}
We emphasize that the absence of a growth gap means that there are infinite-index
subgroups whose growth exponents approach $\omega_S(G)$; it does not require
that any one subgroup attain the growth exponent $\omega_{S}(G)$.

A priori, it may be that the existence of a growth gap depends on the choice
of generating set for $\Gamma$. The question whether having a growth gap is
independent of the finite generating set was posed by Li and Wise
\cite[Problem~9.1]{LiWise2020}, and also appears in
\cite[Problem~1.5]{CoulonLouvarisWiseYehuda2026}. Our first theorem shows that this independence fails in general.

\begin{theorem}\label{thm:intro-product}
Let $G$ be a connected semisimple real Lie group with finite center, no compact factors, real rank at least 2, and with Kazhdan's property (T). Let $\Lambda < G$ be an irreducible lattice. Then the group
\[
  \Lambda \times \Lambda
\]
admits two finite symmetric generating sets $S_1$ and $S_\infty$ such that it
has no growth gap with respect to $S_1$ and has a growth gap with respect to
$S_\infty$.
\end{theorem}

Although the existence of a growth gap may depend on the generating set, its
size cannot be bounded uniformly over all generating sets for the lattices
appearing in Theorem~\ref{thm:intro-product}. In fact, the following more
general statement holds.

\begin{theorem}\label{thm:intro-small-gaps}
Let $\Gamma$ be a finitely generated group containing an infinite-index
nonabelian free subgroup. Then there is a sequence of finite symmetric
generating sets $(S_m)_{m\geq 2}$ of $\Gamma$ such that
\[
  \lim_{m\to\infty}\operatorname{Gap}_{S_m}(\Gamma)=0.
\]
Equivalently,
\[
  \inf\bigl\{\operatorname{Gap}_S(\Gamma):
  S\text{ is a finite symmetric generating set of }\Gamma\bigr\}=0.
\]
In particular, this conclusion holds for
every irreducible lattice $\Lambda<G$ appearing in
Theorem~\ref{thm:intro-product}.
\end{theorem}

In contrast, for hyperbolic groups the existence, rather than the size, of a
growth gap is independent of the finite generating set.

\begin{theorem}\label{thm:intro-hyperbolic}
Let $G$ be a non-elementary hyperbolic group. If $G$ has a growth gap with
respect to one finite symmetric generating set, then it has a growth gap with
respect to every finite symmetric generating set.
\end{theorem}

The article is organized as follows. Theorem~\ref{thm:intro-product} is proved
in Section~\ref{sec:product}. Theorem~\ref{thm:intro-small-gaps} is proved in
Section~\ref{sec:small-gaps}. Theorem~\ref{thm:intro-hyperbolic}, together with a characterization in terms of Property~\emph{(FM)}, is proved in
Section~\ref{sec:hyperbolic}.

\section{A growth gap that depends on the generators}
\label{sec:product}

\subsection{A generating set with a gap for lattices with property (T)}
Let Let $G$ be a connected algebraic semisimple real Lie group with finite center and no compact factors. Let $A$ be a maximal real split torus of $G$ and let $\mathfrak{a}$ denote its Lie algebra. Fix a maximal compact subgroup $K < G$ and a positive closed Weyl chamber $\mathfrak{a}^{+} \subset \mathfrak{a}$ with $A^{+} := \exp \mathfrak{a}^{+}$ so that we have a \emph{Cartan decomposition} $G = KA^{+}K$. Write $\kappa : G \rightarrow \mathfrak{a}^{+}$ for the associated \emph{Cartan projection}, that is, for every $g \in G$, $\kappa(g)$ is the unique element such that $g = k \exp(\kappa(g)) k'$ for some $k,k' \in K$.
\par
Let $X = G/K$ be the Riemannian symmetric space of $G$. Fix a $G$-invariant metric $d_{X}$ induced from a Riemannian metric on $X$; more precisely, fixing a $K$-invariant norm $|| \cdot ||$ on $\mathfrak{a}^{+}$ induced from the Killing form, define $d_{X}(gK,hK) := ||\kappa(g^{-1}h)||$ for all $g,h \in G$. Fix a basepoint $o\in X$. In this subsection, for a discrete subset $H \subset G$, we write 
\begin{align*}
\delta_X(H)=\omega(H,X) = \limsup_{n \to \infty} \frac{1}{n} \log \big| \{ h \in H : d(o, ho) \leq n \} \big|.
\end{align*}
For more details on the material in the above discussion, see \cite[Section 2]{Skenderi2025} or \cite{Knapp1996}.
\begin{lemma}[Equal-length extraction]\label{lem:equal-length}
Let $\Omega=\langle F\rangle^+ \subset G$ be a discrete free subsemigroup of $G$ with a free finite generating set $F \subset G$. Suppose that for some $c,C>0$,
\[
  c|g|_F\le d_X(o,go)\le C|g|_F
  \qquad(g\in\Omega).
\]
If
\[
  0<\beta<\alpha<\delta_X(\Omega),
\]
then, for arbitrarily large integers $r$, there is a finite set
$A_r\subseteq\Omega$ such that:
\begin{enumerate}
\item every element of $A_r$ has the same $F$-word length;
\item $r\le d_X(o,ao)<r+1$ for every $a\in A_r$;
\item $A_r$ freely generates a semigroup;
\item $|A_r|\ge e^{\beta r}$.
\end{enumerate}
\end{lemma}

\begin{proof}
Let
\[
  \mathcal A(r)=\{g\in\Omega:r\le d_X(o,go)<r+1\}.
\]
There are arbitrarily large $r$ such that
$|\mathcal A(r)|\ge e^{\alpha r}$. Otherwise, summing the unit annuli
would imply
\[
  |\{g\in\Omega:d_X(o,go)\le R\}|\ll e^{\alpha R},
\]
contrary to $\delta_X(\Omega)>\alpha$.

Every element of $\mathcal A(r)$ has $F$-word length at most $(r+1)/c$.
Thus only $O(r)$ word lengths occur in $\mathcal A(r)$, and one length layer
$A_r$ satisfies
\[
  |A_r|\ge \frac{e^{\alpha r}}{O(r)}\ge e^{\beta r}
\]
for all sufficiently large $r$.

It remains to prove freeness. The elements of $A_r$ are distinct words over
$F$ of one common length, say $k$. Any equality between two products of
elements of $A_r$ is therefore an equality between two $F$-words. Their
decompositions into consecutive blocks of length $k$ are unique, so the two
sequences of elements of $A_r$ coincide.
\end{proof}

\begin{proposition}\label{prop:base-gap}
Let $G$ be a connected semisimple real Lie group with finite center, no compact factors, and with Kazhdan's property (T). Let $\Lambda < G$ be a lattice. There is a finite symmetric generating set $T$ of $\Lambda$ such
that
\[
  \sup_{\substack{H<\Lambda\\ {}[\Lambda:H]=\infty}}\omega_T(H)
  <\omega_T(\Lambda).
\]
\end{proposition}

\begin{proof}
Put
\[
  v=h_{\mathrm{vol}}(X)=\delta_X(\Lambda).
\]
By the critical exponent gap theorems of Corlette \cite[Theorem 4.4]{Corlette1990} and Leuzinger
\cite[Main Theorem]{Leuzinger2003}, there is $\varepsilon_0 = \varepsilon_0(G)>0$ such that
\begin{equation}\label{eq:leuzinger-gap}
  \delta_X(H)\le v-\varepsilon_0
\end{equation}
for every infinite-index subgroup $H<\Lambda$. We remark that Leuzinger states the theorem for torsion-free discrete groups. To obtain
\eqref{eq:leuzinger-gap}, fix a torsion-free finite-index subgroup
$\Lambda_0<\Lambda$. If $H<\Lambda$ has infinite-index, then
$H_0=H\cap\Lambda_0$ has finite index in $H$ and infinite-index in
$\Lambda_0$. Hence $H_0$ has infinite covolume in $G$. Since critical
exponent is invariant under passage to a finite-index subgroup,
$\delta_X(H)=\delta_X(H_0)\le v-\varepsilon_0$.

Set $q=v-\varepsilon_0$ and choose
\[
  q<\beta<\alpha<\theta<v.
\]
By Borel's Density Theorem \cite{Borel1960}, the lattice $\Lambda$ is Zariski dense in $G$. By
\cite[Theorem~1.3]{Skenderi2025}, there is a finitely generated free
semigroup $\Omega=\langle F\rangle^+\subseteq\Lambda$ such that
$\delta_X(\Omega)\ge\theta$. Moreover, this semigroup is $P$-Anosov in the sense of Kassel--Potrie \cite{KasselPotrie2022}; in fact, more is true:
there is $b>0$ so that, for all $g \in \Omega$, we have
\begin{equation}\label{eq:anosov-linear}
  \min_{\phi\in\Delta}\phi(\kappa(g))\ge b|g|_F,
\end{equation}
where $\Delta$ is the set of simple restricted roots and $\kappa$ is the
Cartan projection. Since every $\phi\in\Delta$ is bounded above by a constant
multiple of $\|\kappa(g)\|=d_X(o,go)$, equation
\eqref{eq:anosov-linear} gives a lower linear bound for $d_X(o,go)$ in terms
of $|g|_F$. The reverse bound follows from the triangle inequality and the
finiteness of $F$. Lemma~\ref{lem:equal-length} therefore applies.

Fix a finite symmetric generating set $U$ of $\Lambda$ and put
\[
  D_U=\max_{u\in U}d_X(o,uo).
\]
Choose $r$ and $A_r$ as in Lemma~\ref{lem:equal-length}, with $r$ large
enough that
\begin{equation}\label{eq:r-choice}
  r+1\ge D_U
  \qquad\text{and}\qquad
  \beta r>q(r+1).
\end{equation}
Define
\[
  T=U\cup A_r\cup A_r^{-1}.
\]
Since $A_r$ freely generates a semigroup, all $|A_r|^n$ positive words of
length $n$ give distinct elements of $B_T(n)$. Hence
\begin{equation}\label{eq:ambient-lower}
  \omega_T(\Lambda)\ge\log|A_r|\ge\beta r.
\end{equation}

Every element of $T$ moves $o$ by at most $r+1$. Therefore
\[
  H\cap B_T(n)
  \subseteq
  \{h\in H:d_X(o,ho)\le(r+1)n\}.
\]
For every infinite-index $H<\Lambda$, equations
\eqref{eq:leuzinger-gap}--\eqref{eq:ambient-lower} give
\[
  \omega_T(H)
  \le(r+1)\delta_X(H)
  \le q(r+1)
  <\beta r
  \le\omega_T(\Lambda).
\]
The upper bound is uniform in $H$.
\end{proof}

\subsection{The two product generating sets}

\begin{proof}[Proof of Theorem~\ref{thm:intro-product}]
Let $T$ be supplied by Proposition~\ref{prop:base-gap}. Write
\[
  h=\omega_T(\Lambda)
\]
and choose $\eta>0$ such that
\begin{equation}\label{eq:base-word-gap}
  \omega_T(J)\le h-\eta
\end{equation}
for every infinite-index subgroup $J<\Lambda$. Since the trivial subgroup
has exponent zero, we may take $\eta\le h$.

First define
\[
  S_1=(T\times\{1\})\cup(\{1\}\times T).
\]
Then
\begin{equation}\label{eq:l1-length}
  |(g_1,g_2)|_{S_1}=|g_1|_T+|g_2|_T.
\end{equation}
We claim that $\omega_{S_1}(\Lambda\times\Lambda)=h$. If
$b_n=|B_T(n)|$, then
\[
  |B_{S_1}(n)|\le\sum_{i+j\le n}b_i b_j.
\]
For every $\epsilon>0$, there is $C_\epsilon>0$ such that
$b_k\le C_\epsilon e^{(h+\epsilon)k}$ for every $k\ge0$. Consequently
\[
  |B_{S_1}(n)|
  \le C_\epsilon^2(n+1)^2e^{(h+\epsilon)n}.
\]
This proves the upper bound, and the reverse bound follows from
$B_T(n)\times\{1\}\subseteq B_{S_1}(n)$. The infinite-index subgroup
$\Lambda\times\{1\}$ therefore satisfies
\[
  \omega_{S_1}(\Lambda\times\{1\})
  =h
  =\omega_{S_1}(\Lambda\times\Lambda).
\]
Thus there is no growth gap for $S_1$.

For the second generating set put $T^*=T\cup\{1\}$ and
\[
  S_\infty=(T^*\times T^*)\setminus\{(1,1)\}.
\]
Then
\begin{equation}\label{eq:linfty-length}
  |(g_1,g_2)|_{S_\infty}
  =\max\{|g_1|_T,|g_2|_T\},
\end{equation}
so
\begin{equation}\label{eq:linfty-ball}
  B_{S_\infty}(n)=B_T(n)\times B_T(n),
  \qquad
  \omega_{S_\infty}(\Lambda\times\Lambda)=2h.
\end{equation}

Let $H<\Lambda\times\Lambda$ have infinite-index, let $p_i$ be the two
coordinate projections, and set $P_i=p_i(H)$. If, say, $P_1$ has infinite
index in $\Lambda$, then
\[
  |H\cap B_{S_\infty}(n)|
  \le |P_1\cap B_T(n)|\,|B_T(n)|.
\]
It follows from \eqref{eq:base-word-gap} that
\begin{equation}\label{eq:projection-case}
  \omega_{S_\infty}(H)
  \le\omega_T(P_1)+h
  \le2h-\eta.
\end{equation}
The same argument applies if $P_2$ has infinite-index.

It remains to assume that both $P_1$ and $P_2$ have finite index in
$\Lambda$. Define
\[
  N_1=\{x\in P_1:(x,1)\in H\},
  \qquad
  N_2=\{y\in P_2:(1,y)\in H\}.
\]
Then $N_i\lhd P_i$. Moreover, there is an isomorphism
\begin{equation}\label{eq:goursat}
  P_1/N_1\longrightarrow P_2/N_2,
  \qquad
  xN_1\mapsto yN_2
  \quad\text{when }(x,y)\in H.
\end{equation}
Indeed, changing $y$ changes it by an element of $N_2$, and the definitions
of $P_i$ and $N_i$ give surjectivity and injectivity. This is the usual
Goursat description, included here to fix the relevant kernels.

Each $P_i$ is a finite-index subgroup of the higher-rank irreducible lattice $\Lambda$. By the Margulis normal
subgroup theorem \cite{Margulis1991}, each $N_i$ is either finite or has finite index
in $P_i$. The isomorphism \eqref{eq:goursat} shows that $N_1$ has finite
index exactly when $N_2$ does. If both had finite index, then
$N_1\times N_2\subseteq H$ would have finite index in
$P_1\times P_2$, forcing $H$ to have finite index in
$\Lambda\times\Lambda$. Therefore both $N_1$ and $N_2$ are finite.

The kernel of $p_1|_H$ is $\{1\}\times N_2$. Hence
\[
  |H\cap B_{S_\infty}(n)|
  \le |N_2|\,|P_1\cap B_T(n)|
  \le |N_2|\,|B_T(n)|,
\]
and consequently
\begin{equation}\label{eq:subdirect-case}
  \omega_{S_\infty}(H)\le h.
\end{equation}
Equations \eqref{eq:projection-case} and \eqref{eq:subdirect-case}, together
with $\eta\le h$, show uniformly that
\[
  \omega_{S_\infty}(H)
  \le2h-\eta
  <2h
  =\omega_{S_\infty}(\Lambda\times\Lambda).
\]
Thus $S_\infty$ has a growth gap.
\end{proof}

\section{Arbitrarily small growth gaps}
\label{sec:small-gaps}

The construction in this section is elementary and applies to any finitely
generated group containing an infinite-index nonabelian free subgroup.

\begin{proof}[Proof of Theorem~\ref{thm:intro-small-gaps}]
Fix a finite symmetric generating set $U$ of $\Gamma$, and write
\[
  q=|U|.
\]
Let
\[
  F=\langle a,b\rangle<\Gamma
\]
be a free subgroup of rank two and of infinite-index. For every $m\geq 2$,
put
\[
  a_i=b^iab^{-i}\qquad(0\leq i\leq m-1),
\]
and define
\[
  A_m=\{a_0,\ldots,a_{m-1}\},
  \qquad
  H_m=\langle A_m\rangle.
\]
The elements of $A_m$ freely generate $H_m$. Indeed, consider the
homomorphism
\[
  F\longrightarrow\Z,\qquad a\mapsto 0,\quad b\mapsto 1.
\]
The Schreier rewriting process gives
\[
  \{b^iab^{-i}:i\in\Z\}
\]
as a free basis of its kernel. In particular, $H_m$ is free of rank $m$.
Moreover, since $H_m<F$ and $F$ has infinite-index in $\Gamma$, the subgroup
$H_m$ also has infinite-index in $\Gamma$.

Now let
\[
  S_m=U\cup A_m\cup A_m^{-1}.
\]
This is a finite symmetric generating set of $\Gamma$. Every freely reduced
word of length $n$ in the alphabet $A_m\cup A_m^{-1}$ represents a distinct
element of $H_m\cap B_{S_m}(n)$. Hence
\[
  |H_m\cap B_{S_m}(n)|
  \geq 2m(2m-1)^{n-1},
\]
and therefore
\begin{equation}\label{eq:small-gap-subgroup-lower}
  \omega_{S_m}(H_m)\geq\log(2m-1).
\end{equation}

On the other hand, every element of $B_{S_m}(n)$ admits a geodesic
representative in which no letter is immediately followed by its inverse.
Consequently,
\[
  |B_{S_m}(n)|
  \leq
  1+\sum_{k=1}^{n}|S_m|(|S_m|-1)^{k-1}.
\]
Since $|S_m|\leq q+2m$, it follows that
\begin{equation}\label{eq:small-gap-ambient-upper}
  \omega_{S_m}(\Gamma)
  \leq\log(|S_m|-1)
  \leq\log(2m+q-1).
\end{equation}
Since $H_m$ has infinite-index, equations
\eqref{eq:small-gap-subgroup-lower} and
\eqref{eq:small-gap-ambient-upper} give
\[
  0
  \leq\operatorname{Gap}_{S_m}(\Gamma)
  \leq
  \omega_{S_m}(\Gamma)-\omega_{S_m}(H_m)
  \leq
  \log\left(\frac{2m+q-1}{2m-1}\right).
\]
The final expression tends to zero as $m\to\infty$, proving the general
statement.

Finally, let $\Lambda<G$ be as in
Theorem~\ref{thm:intro-product}. By Borel's Density Theorem
\cite{Borel1960}, the group $\Lambda$ is not virtually solvable. The Tits
alternative \cite{Tits1972} therefore implies that $\Lambda$ contains a
nonabelian free subgroup $F$. Since $G$ has Kazhdan's property~\emph{(T)},
so does its lattice $\Lambda$. The subgroup $F$ must have infinite-index:
otherwise it would inherit Property~\emph{(T)}, whereas a nonabelian free
group does not have Property~\emph{(T)}. Thus the general statement applies
to $\Lambda$, as desired.
\end{proof}

\begin{remark}
Theorem~\ref{thm:intro-small-gaps} does not assert that one of the generating
sets $S_m$ has zero gap. It says that no positive constant can bound
$\operatorname{Gap}_S(\Gamma)$ uniformly as the finite symmetric generating
set $S$ varies.
\end{remark}

\section{Generator independence for hyperbolic groups}
\label{sec:hyperbolic}

An action of a countable group $G$ on a countable set $Y$ is
\emph{amenable} if there is a $G$-invariant mean on $\ell^\infty(Y)$.
Equivalently, the permutation representation
\[
  \lambda_Y:G\longrightarrow\cU(\ell^2(Y))
\]
has almost invariant unit vectors. Following Glasner--Monod
\cite{GlasnerMonod2007}, $G$ has \emph{Property~(FM)} if every amenable
action of $G$ on a countable set has a finite orbit.

\begin{theorem}\label{thm:hyperbolic-characterization}
Let $G$ be a non-elementary hyperbolic group. The following are equivalent.
\begin{enumerate}
\item There is a finite symmetric generating set $S$ and $\varepsilon>0$
such that
\[
  \omega_S(H)\le\omega_S(G)-\varepsilon
\]
for every infinite-index subgroup $H<G$.
\item The same assertion holds for every finite symmetric generating set of
$G$.
\item The group $G$ has Property~\emph{(FM)}.
\end{enumerate}
More generally, suppose that $G$ acts properly and cocompactly by isometries
on a proper geodesic hyperbolic space $X$. Then
\[
  \sup_{[G:H]=\infty}\omega(H,X)<\omega(G,X)
\]
if and only if $G$ has Property~\emph{(FM)}.
\end{theorem}

The implication from a growth gap to Property~(FM) requires a uniform form
of the spherical estimate from \cite[Appendix~B]{CoulonDalboSambusetti2018}.

In what follows, if $\mu$ is a finitely supported probability measure on $G$, we use the same
notation for the associated averaging operator
\[
  \lambda_Y(\mu)
  =
  \sum_{g\in G}\mu(g)\lambda_Y(g)
  \in\mathcal B(\ell^2(Y)).
\]
Thus $\|\lambda_Y(\mu)\|$ denotes the operator norm of this averaging operator
on $\ell^2(Y)$. In particular, when $Y=G/H$, we write
$\lambda_{G/H}(\mu)$ for the corresponding operator in the quasi-regular
representation.

\begin{lemma}\label{lem:spherical-contraction}
Let $G$ act properly and cocompactly by isometries on a proper geodesic
hyperbolic space $X$, and assume that $G$ is non-elementary. Put
$\omega=\omega(G,X)$. There are constants $a,D>0$ and, for every
$\ell\ge a$, a symmetric finitely supported probability measure $\mu_\ell$
on $G$ such that the following holds. For every $\eta>0$ and every subgroup
$H<G$ satisfying
\[
  \omega(H,X)\le\omega-\eta,
\]
one has
\begin{equation}\label{eq:spherical-contraction}
  \|\lambda_{G/H}(\mu_\ell)\|
  \le
  D\left(\frac{\ell}{a}+1\right)e^{-c\ell},
  \qquad
  c=\frac12\min\{\omega,\eta\}.
\end{equation}
Here $\lambda_{G/H}$ denotes the quasi-regular representation on
$\ell^2(G/H)$.
\end{lemma}

\begin{proof}
Fix $o\in X$. After increasing the thickness $a$, let $\mu_\ell$ be the
uniform probability measure on
\[
  \mathcal S_\ell
  =\{g\in G:\ell-a<d_X(o,go)\le\ell\}.
\]
The proof of \cite[Proposition~B.3]{CoulonDalboSambusetti2018} gives a
constant $D>0$, independent of $H,\ell,n$, and $g$, such that
\begin{equation}\label{eq:pointwise-convolution}
  \mu_\ell^{*n}(g)
  \le
  \left[D\left(\frac{\ell}{a}+1\right)\right]^n
  \exp\left(-\frac{\omega}{2}
  \bigl(n\ell+d_X(o,go)\bigr)\right).
\end{equation}
Write $\omega_H=\omega(H,X)$. For every $\alpha>\omega_H$, there is
$A=A(H,\alpha)$ such that
$|H\cap B_X(o,r)|\le Ae^{\alpha r}$ for all $r\ge0$. Partitioning $H$ into
annuli of width $a$ and summing \eqref{eq:pointwise-convolution} gives
\begin{equation}\label{eq:subgroup-return}
  \mu_\ell^{*n}(H)
  \le
  A'
  \left[D\left(\frac{\ell}{a}+1\right)\right]^n
  e^{-\omega n\ell/2}
  \sum_{0\le ka\le n\ell}e^{(\alpha-\omega/2)ka},
\end{equation}
where $A'$ may depend on $H$ and $\alpha$, but not on $n$. Since
$\mu_\ell$ is symmetric, the return-probability formula gives
\begin{equation}\label{eq:return-radius}
  \|\lambda_{G/H}(\mu_\ell)\|
  =
  \lim_{n\to\infty}
  \bigl(\mu_\ell^{*2n}(H)\bigr)^{1/(2n)}.
\end{equation}

If $\omega_H<\omega/2$, choose
$\omega_H<\alpha<\omega/2$. The sum in
\eqref{eq:subgroup-return} is bounded independently of $n$, and hence
\[
  \|\lambda_{G/H}(\mu_\ell)\|
  \le
  D\left(\frac{\ell}{a}+1\right)e^{-\omega\ell/2}.
\]
If $\omega_H\ge\omega/2$, take $\alpha=\omega_H+\eta/2$. The geometric sum
in \eqref{eq:subgroup-return} has exponential rate
$\alpha-\omega/2$, and therefore
\[
  \|\lambda_{G/H}(\mu_\ell)\|
  \le
  D\left(\frac{\ell}{a}+1\right)
  e^{(\omega_H+\eta/2-\omega)\ell}
  \le
  D\left(\frac{\ell}{a}+1\right)e^{-\eta\ell/2}.
\]
Combining the two cases proves \eqref{eq:spherical-contraction}.
\end{proof}

\begin{proof}[Proof of Theorem~\ref{thm:hyperbolic-characterization}]
We first prove the more general assertion concerning $X$.

Suppose that for some $\eta>0$,
\begin{equation}\label{eq:gap-X}
  \omega(H,X)\le\omega(G,X)-\eta
\end{equation}
for every infinite-index subgroup $H<G$. Let $G\curvearrowright Y$ be an
amenable action on a countable set. Assume for a contradiction that it has
no finite orbit. Choosing one point in each orbit and writing $H_i$ for its
stabilizer, we have
\[
  Y=\bigsqcup_{i\in I}G/H_i,
  \qquad [G:H_i]=\infty,
\]
and
\[
  \lambda_Y=\bigoplus_{i\in I}\lambda_{G/H_i}.
\]
Lemma~\ref{lem:spherical-contraction} and \eqref{eq:gap-X} imply that, for
all sufficiently large $\ell$,
\begin{equation}\label{eq:Y-contraction}
  \|\lambda_Y(\mu_\ell)\|
  =
  \sup_{i\in I}\|\lambda_{G/H_i}(\mu_\ell)\|
  <1.
\end{equation}
On the other hand, amenability gives almost invariant unit vectors in
$\ell^2(Y)$. Since $\mu_\ell$ has finite support, this forces
$\|\lambda_Y(\mu_\ell)\|=1$, contradicting
\eqref{eq:Y-contraction}. Thus $G$ has Property~(FM).

Conversely, suppose that $G$ has Property~(FM). A proper cocompact action on
a proper geodesic hyperbolic space is strongly positively recurrent.
Therefore \cite[Theorem~6.15]{CoulonDougallSchapiraTapie2025} gives
$\theta>0$ such that, for every subgroup $H<G$,
\[
  \omega(H,X)>(1-\theta)\omega(G,X)
  \quad\Longrightarrow\quad
  [G:H]<\infty.
\]
Since $G$ is non-elementary, $\omega(G,X)>0$. Every infinite-index subgroup
therefore satisfies
\[
  \omega(H,X)\le(1-\theta)\omega(G,X),
\]
which is a uniform growth gap.

Finally take $X=\operatorname{Cay}(G,S)$. For every finite symmetric
generating set $S$, this is a proper geodesic hyperbolic space on which $G$
acts properly and cocompactly, and $\omega(H,X)=\omega_S(H)$. The
equivalence with Property~(FM), which does not refer to $S$, proves the
equivalence of (1), (2), and (3).
\end{proof}

\begin{remark}\label{rem:coamenability-not-enough}
The pointwise theorem of \cite{CoulonDalboSambusetti2018}, stating that
$\omega(H,X)=\omega(G,X)$ if and only if $H$ is co-amenable in $G$, does not
by itself prove Theorem~\ref{thm:hyperbolic-characterization}. It excludes
attainment of the endpoint, but not a sequence of infinite-index subgroups
whose exponents approach it. The uniform estimate in
Lemma~\ref{lem:spherical-contraction} is what promotes a growth gap for one
metric to Property~(FM).
\end{remark}

\begin{remark}\label{rem:scope}
The supremum throughout this paper is over all infinite-index subgroups. If
one restricts the spectrum to finitely generated subgroups, the implication
from a growth gap to Property~(FM) does not follow from this proof, since
stabilizers of an arbitrary amenable action need not be finitely generated.
Also, if $G$ is infinite and virtually cyclic, then $\omega_S(G)=0$ and the
trivial subgroup has the same exponent, so no finite generating set has a
growth gap.
\end{remark}

\bibliographystyle{alpha}
\bibliography{bib}

\end{document}